\documentclass[10pt, reqno]{amsart}
\usepackage{amsmath}
\usepackage{amsfonts}
\usepackage{amssymb}
\usepackage{amsthm}

\newtheorem{theorem}{Theorem}

\newtheorem{proposition}[theorem]{Proposition}
\newtheorem{definition}[theorem]{Definition}
\newtheorem{lemma}[theorem]{Lemma}
\newtheorem{corollary}[theorem]{Corollary}

\theoremstyle{remark}

\newtheorem{note}{{\bf Note}}
\newcommand{\const}{C}

\newcommand{\tr}{\mathrm{Tr}}

\newcommand{\supp}[1]{{\text{supp}}({#1})}

\newcommand{\nat}{\mathbb{N}}
\newcommand{\reals}{\mathbb{R}}

\newcommand{\dif}{\mathrm{d}}

\newcommand{\dist}{\mathrm{dist}}

\renewcommand{\a}{\mathbf{a}}

\renewcommand{\L}{\mathcal{L}}
\newcommand{\zb}[1]{\ensuremath{\boldsymbol{#1}}}
\newcommand{\bx}{\zb x}
\newcommand{\bxi}{\zb \xi}
\newcommand{\bzeta}{\zb \zeta}
\newcommand{\balpha}{\zb \alpha}

\begin{document}
\title{Surface Spline Approximation on SO(3)}

\author{Thomas Hangelbroek}
\address{Department of Mathematics, Vanderbilt University, Nashville, TN 37240, USA. } 
\email{thomas.c.hangelbroek@vanderbilt.edu}
\thanks{Thomas Hangelbroek wass supported by an NSF Postdoctoral Research Fellowship.}
\author{Dominik Schmid}
 \address{Institute of Biomathematics and Biometry, Helmholtz Zentrum M\"unchen,\\ German Research Center for Environmental Health, Ingolst\"adter Landstra{\ss}e 1,\\ D-85764 Neuherberg, Germany.}
\email{dominik.schmid@helmholtz-muenchen.de}
\thanks{Dominik Schmid is supported by Deutsche Forschungsgemeinschaft Grant PO 711/9-1 and FI 883/3-1.}

\begin{abstract}
The purpose of this article is to introduce a new class of kernels on $SO(3)$ for approximation and
interpolation, and to estimate the approximation power of the associated spaces. 
The kernels we consider arise as linear combinations of Green's functions of certain differential operators on the rotation group. 
They are conditionally positive definite and have a simple, closed-form expression, lending themselves to direct implementation via, e.g., interpolation or least-squares approximation.
To gauge the approximation power of the underlying spaces, we introduce an approximation scheme providing precise $L_p$ error estimates for linear schemes, namely with $L_p$ approximation order conforming to the $L_p$ smoothness of the target function.
\end{abstract}

\maketitle
\section{Introduction}
Approximating a function $f:\mathbb{X}\to\mathbb{C}$ on various structures $\mathbb{X}$ by linear combinations of translates of a given single basis function is a widely used method. 
At the heart of this methodology is the creation of an approximant using a fixed kernel 
$\kappa: \mathbb{X}\times\mathbb{X}\to\mathbb{C}$,
by taking a scattered linear combination of copies of this kernel:
$s_{\zb \Xi}({\bx}) = \sum_{\bxi \in \zb \Xi} A_{\bxi} \kappa(\bx, \bxi).$
 The success of these methods derive from their ability to generate approximants from  data having arbitrary geometry.
The case where the underlying structure are Euclidean spaces $\mathbb{R}^d$ or Euclidean spheres $\mathbb{S}^d$ has been studied in great detail over the last decade, 
see \cite{Buhm,Wend}.
Such kernel methods have found success in approximation theory (for treating high dimensional scattered data) and learning theory (in kernel based learning, support vector machines, neural networks, etc).

However, in various applications we are confronted with the situation that the underlying set is a compact or locally compact group and one is asked to propose suitable approximation procedures on these structures (this is a setting with no natural dilation operator and no `regular grids').
Such problems arise in biochemistry, crystallography and robotics to name only a few. The monograph \cite{ChKy00} provides a great collection of scattered data approximation problems where different matrix groups are involved.

In view of applications, the rotation group $SO(3)$ is, without doubt, one of the most important groups.
There exist a wealth of scattered data approximation problems on $SO(3)$, cf. \cite{Bo07,SHFPP07,CSB05,YLV04}.
Such problems have attracted significant recent attention from the mathematical community, which has investigated approximation via Wigner functions
\cite{GrPo09,PoPrVo09,GrKu08,KoRe08,Schmid1}, and kernels \cite{Gu96,Schmid2}.

The goal of this article is to introduce a family of kernels on the rotation group suitable for approximation and to investigate the approximation power of this family by estimating
the performance of a (theoretical) approximation scheme.
The kernels we present are notable in many respects: 
\begin{itemize} 
\item they are {\em conditionally  positive definite}, have a simple, closed form expression, and are consequently well-suited for implementation (for, say, interpolation or least squares approximation); 
\item they invert iterated, perturbed Laplace \!--\! Beltrami operators, tying them directly to successful kernel based approximation in other settings (cf.  {\it surface spline} approximation on $\reals^d$ \cite{D1},\cite{D2}, restricted surface splines on spheres \cite{Hang}, and, of course,  polynomial splines in one dimension); 
\item they have excellent scaling properties, much like the surface splines in the Euclidean case, suggesting that they can adjust themselves to data whose density varies spatially, as in \cite{DeRo}.
\end{itemize}
In addition to these three points, we study the approximation power of the spaces
generated by the kernels. We develop an approximation scheme and accompanying error estimates derived from the kernel's role as a Green's function. 
To be sure,  the scheme we develop is theoretical and does not immediately lend itself to direct implementation, but the accompanying error estimates are of a type that is notoriously difficult for kernel based approximation when the kernel is not dilated. This is true even in the Euclidean case, as in radial basis function (RBF) approximation. 
The results we are after give precise error rates for functions at all appropriate levels of smoothness; we determine $L_p$ rates of convergence commensurate with the $L_p$ smoothness of the function being approximated. 
We note that the results we seek are well
known for kernel approximation where dilations and regular grids are employed 
 \cite{BDR, LJCh},
and even without grids, \cite{Kyri, Jia}.
However, in the absence of dilation, error estimates are often known only for target functions residing in a reproducing  kernel Hilbert space.

We have organized the article as follows. In the next section we give the necessary results of analysis on the rotation group. 
In Section 3 we introduce the kernels and give their expansions in terms of Wigner D-functions,  the basic elements of Fourier analysis on the group. From this we are able to show that the kernels are conditionally positive definite, and, hence are capable of solving interpolation and other scattered data fitting problems. The expansion also allows us to identify the integral identities these kernels satisfy.
Section 4 establishes the basic strategy to localize a kernel: by replacing a given kernel by a linear combination of its nearby copies. A precise estimate of the error in making this exchange is then given.
In Section 5 we propose an approximation scheme and present error estimates that demonstrate the approximation power of the kernels. Finally, Section 6 involves a discussion of potential generalizations of this method to other groups and homogeneous spaces. Recent work 
on bounding Lebesgue constants for kernel interpolation is discussed, which indicates that the theoretical error estimates may also hold for more practical schemes like interpolation and least squares approximation.

\section{Background}

%
%
%
In this section, we collect some basic material and necessary notation on the rotation group to keep the paper self-contained.
\subsection{Basic Facts}\label{ss_basic_facts}
Let $SO(3):=\{\zb{x}\in\mathbb{R}^{3\times 3}:\zb{x}^T\zb{x}=\zb{e}, {\rm det} \,\zb{x}=1\}$ denote the non-Abelian compact group of proper rotations in the Euclidean space $\mathbb{R}^3$ 
and let $\mu$ denote the normalized Haar measure on $SO(3)$, i.e., we have $\mu(SO(3))=1$.
There are various ways to parameterize this group. 
An important parametric form for the rotation group uses the so-called Euler angles $(\varphi_1,\theta,\varphi_2)\in[0,2\pi)\times[0,\pi]\times[0,2\pi)$.
In the literature there are miscellaneous conventions of Euler angles. In this paper we follow the conventions made in \cite[Section 1.2]{GMS63}.
There the Euler angles are defined such that we can write every $\bx\in SO(3)$ as
\begin{equation}\label{defi_Euler_angles}
\zb{x}=\zb{x}(\varphi_1,\theta,\varphi_2)=\zb{s}_z[\varphi_1]\zb{s}_x[\theta]\zb{s}_z[\varphi_2],
\end{equation}
where
$$
\zb{s}_z[t]=\left(\begin{array}{ccc}\cos\, t&-\sin\, t&0\\\sin\, t&\cos\, t&0\\ 0&0&1 \end{array}\right),\qquad 
\zb{s}_x[t]=\left(\begin{array}{ccc}1&0&0\\0&\cos\, t&-\sin\, t\\ 0&\sin\, t&\cos\, t  \end{array}\right)
$$
are rotations about the $z$-axis and the $x$-axis, respectively.

Any integral over $SO(3)$ (or portions thereof) will always be assumed to be the Haar integral. 
Using Euler angles, the Haar integral of a function $f$ on $SO(3)$ reads as 
\begin{equation}\label{haar-integral}
\int_{SO(3)}f(\zb{x}){\dif}\mu(\zb{x})=\frac{1}{8\pi^2}\int_0^{2\pi}\int_0^\pi\int_0^{2\pi} \emph{f} \,(\varphi_1,\theta,\varphi_2)\ \sin\theta\,{\dif}\varphi_1\,{\dif}\theta\,{\dif}\varphi_2.
\end{equation}

Besides the well-known parameterization via Euler angles, the parameterization via the three-dimensional projective space is also of some importance to us.
Let $\mathcal{K}_\pi$ be the closed ball in $\mathbb{R}^3$ of radius $\pi$ centered at the origin and identify antipodal points. This is the three-dimensional projective space. An element $\zb{x}\in SO(3)$ is identified with a point in the projective space $\mathcal{K}_\pi$ by $\zb{x}\to \omega{ \bf r}$ where ${ \bf r}$, satisfying $\zb{x}{ \bf r}={ \bf r}$ and $\|{ \bf r}\|=1$, is the {\em rotation axis} and $\omega$, which can be chosen in $[0,\pi]$, is the {\em rotation angle} of $\zb{x}$. 

Then it is easy to see that 
\begin{equation}\label{metric}
\dist(\zb{x},\zb{y}):=\omega(\zb{y}^{-1}\zb{x})
\end{equation}
defines a translation invariant metric on $SO(3)$.
Of course, this is not the only metric one can place on 
$SO(3)$, see \cite{Mitch} for a discussion of other metrics and sampling in $SO(3)$, but it is compatible with the Riemannian structure of the group and we will use this metric throughout this article. 
This leads to the usual definition of the ball centered at $\balpha$ having radius $\rho$: 
$B(\balpha,\rho) = \{\bzeta\mid \dist(\bzeta,\balpha)<\rho)$.
A consequence of
(\ref{haar-integral}) is that we can estimate the volume of a ball
$\mu\bigl(B(\balpha,\rho)\bigr) = \int_{SO(3)}\mathbf{1\hspace{-1.2mm}l}_{B(\balpha,\rho)}(\bx)\dif \mu(\bx)$, 
obtaining constants $b_1:= 2/(3\pi^3)$ and $b_2 := 1/(6\pi)$ so that for any $0<\rho<\pi$ and 
any $\balpha\in SO(3)$,
$$b_1 \rho^3 \le \mu\bigl(B(\balpha,\rho)\bigr) \le b_2 \rho^3.$$
We utilize the notation $\omega$ in the context of the rotation angle and $\dist$ in the context of the metric. 

We can use this metric in order to quantify the distribution of the centers in the set
$$
\zb\Xi:=\{\bxi_j\in SO(3): j=0,\ldots,M-1\}.
$$
To do so we define two parameters. The first one is the separation distance
$$
q=q(\zb\Xi):=\min_{0\leq i< j\leq M-1}\dist(\bxi_i,\bxi_j)
$$
which measures `clustering trends' in $\zb\Xi\subset SO(3)$. 
On the other hand, the fill distance
$$
h=h(\zb\Xi):=\max_{\zb y\in SO(3)}\min_{j=0,\ldots,M-1}\dist(\bxi_j,\zb y)
$$
describes the density of $\zb\Xi$ in $SO(3)$.

\subsection{Harmonic Analysis and Kernels}
The basic building blocks in harmonic analysis on topological groups are the irreducible unitary representations.
On $SO(3)$, it is possible to compute a complete set of irreducible unitary representations, explicitly.
More precisely, using the quasi-left regular representation of $SO(3)$ on $L_2(\mathbb{S}^2)$ one is able to compute for each $\ell\in\mathbb{N}_0$ a $(2\ell+1)\times(2\ell+1)$-matrix $D^\ell$ 
such that these operators $D^\ell,\ell\in\mathbb{N}_0$, form a complete set of irreducible unitary representations of  $SO(3)$.
The matrix elements $D_{k,m}^\ell$ in the canonical basis are often called Wigner-D functions of degree $\ell$ and orders $k$ and $m$. In terms of Euler angles these functions are given by
\begin{equation}\label{explicit_form_matrix_elements_SO(3)}
D_{k,m}^\ell(\bx)=D_{k,m}^\ell(\varphi_1,\theta,\varphi_2)=e^{-ik\varphi_1}P_{k,m}^\ell(\cos\theta)e^{-im\varphi_2},\quad -\ell\leq k,m\leq \ell,
\end{equation}
where the function $P_{k,m}^\ell$ is given by 
$$
P_{k,m}^\ell(t)=C(1-t)^{-\frac{m-k}{2}}(1+t)^{-\frac{m+k}{2}}\frac{d^{\ell-m}}{dt^{\ell-m}}\Big((1-t)^{\ell-k}(1+t)^{\ell+k}\Big)
$$
with $C=\frac{(-1)^{\ell-k}i^{m-k}}{2^\ell(\ell-k)!}\sqrt{\frac{(\ell-k)!(\ell+m)!}{(\ell+k)!(\ell-m)!}}$.

We refer to \cite{GMS63} for all the details in these computations.
By means of the Peter-Weyl Theorem the Wigner-D functions, the matrix elements of the irreducible representations, constitute an orthogonal basis of the $L_2(SO(3))$. 
Hence, every $f \in L_2(SO(3))$ can be expanded in a $SO(3)$ Fourier series
$$
f = \sum_{\ell \in \mathbb{N}_0} \sum_{k,m=-\ell}^\ell\sqrt{2\ell+1}\, \hat{f}_{k,m}^\ell D_{k,m}^\ell
$$
with $SO(3)$ Fourier coefficients
$  \hat f_{k,m}^\ell= \sqrt{2\ell+1} \int_{SO(3)} f (\bx) \overline{D_{k,m}^\ell(\bx)} d\mu(\bx).$
On the other hand, to each eigenvalue, $\nu_{\ell} := \ell(\ell+1)$, of the Laplace \!--\! Beltrami operator $\Delta$ on  $SO(3)$, there corresponds an eigenspace of dimension 
$(2\ell+1)^2$ 
with orthogonal basis $(D_{k,m}^{\ell})_{k,m}$, i.e. we have for all $\ell\in\mathbb{N}_0$ and all $k,m=-\ell,\ldots,\ell$ that
\begin{equation}\label{LaBe}
\Delta\,D_{k,m}^\ell=\ell(\ell+1)D_{k,m}^\ell.
\end{equation}
We denote this eigenspace by $\mathcal{H}_{\ell},$ and we refer to  ${\zb \Pi}_n =\bigoplus_{\ell = 0}^n \mathcal{H}_{\ell}$ as the polynomials of degree no more than $n$.

The kernels we study in this article possess a certain symmetry; they are functions of the distance between the two arguments, 
${\rm dist}(\bx,\zb\alpha)=\omega(\balpha^{-1}\bx)\in [0,\pi]$ only. 
By the bijectivity of $\sin\left( \frac{\cdot}{2}\right)$ and $\cos\left( \frac{\cdot}{2}\right)$ on $[0,\pi]$ we are entitled to express these kernels as $k(\bx,\balpha) = \phi(\sin ( \frac{\omega(\balpha^{-1} \bx)}{2})) = \psi(\cos( \frac{\omega(\balpha^{-1} \bx)}{2})) $, where $\phi, \psi :[0,1] \to \reals$. There are benefits to both alternatives, as we shall see in the following sections.

On $SO(3)$ one easily checks that the functions that only depend on the rotational angle of the argument coincide exactly with the class functions -- functions that are constant on conjugacy classes, i.e. $f(\bx)=f(\zb y\bx \zb y^{-1})$ for all $\bx,\zb y\in SO(3)$. 
In other words, for any class function $f$ on $SO(3)$ there is a uniquely determined $\widetilde{f}:[0,\pi]\to\mathbb{C}$ such that $f(\bx)=\widetilde{f}(\omega(\bx))$ for all $\bx\in SO(3)$.
The Haar integral for such a class function reads as
\begin{equation}\label{Haar_integral_class_SO(3)}
\int_{SO(3)}f(\bx)\dif\mu(\bx)=\int_{SO(3)}\widetilde{f}(\omega(\bx))\dif\mu(\bx)=\frac{2}{\pi}\int_0^\pi\widetilde{f}(t)\sin^2\left(\frac{t}{2}\right)\dif t.
\end{equation}

We make the convention that if $f$ is a class function on $SO(3)$ and if no confusion occurs, we subsequently drop the tilde.

Since the characters $\mathfrak{c_{\ell}} = \tr D^{\ell}, \, \ell \in \nat_0,$ of $SO(3)$ form an orthonormal basis for the space of class functions in $L_2(SO(3))$,
we can decompose any continuous kernel possessing the symmetry mentioned above as
$k(\bx,\balpha) = \sum _{\ell = 0}^{\infty} \widetilde{k}(\ell) \mathfrak{c_{\ell}}(\balpha^{-1}\bx)$.
Using the formula 
$\mathfrak{c}_{\ell} (\balpha^{-1}\bx)=   {\mathcal U}_{2\ell} \bigl(\cos(\frac{\omega(\balpha^{-1}\bx)}{2})\bigr)$, 
such kernels can be expressed as an expansion of even powered Chebyshev polynomials of the second kind,
\begin{equation}\label{Cheb}
k(\bx,\balpha) = \sum _{\ell = 0}^{\infty} \widetilde{k}(\ell) \mathfrak{c_{\ell}}(\balpha^{-1}\bx)= \sum_{\ell = 0}^{\infty} \widetilde{k}(\ell) {\mathcal U}_{2\ell} \left(\cos\left(\frac{\omega(\balpha^{-1}\bx)}{2}\right)\right).\end{equation}
Alternatively, we can write:
%
%
\begin{eqnarray}
k(\bx,\zb  \alpha) 
&=&  \sum_{\ell = 0}^{\infty} 
 \widetilde{k}(\ell) {\mathfrak c}_{\ell}(\balpha^{-1}\bx)
=
 \sum_{\ell = 0}^{\infty}
\widetilde{k}(\ell) \tr(D^{\ell}(\balpha^{-1}\bx))\nonumber\\
 &=& 
 \sum_{\ell = 0}^{\infty}
\widetilde{k}(\ell) \tr\left({D^{\ell}(\balpha)}^{*}\,D^{\ell}(\bx)\right)\nonumber\\
&=&
\sum_{\ell = 0}^{\infty}\sum_{\iota=-\ell}^{\ell} \sum_{\nu=-\ell}^{\ell} 
\widetilde{k}(\ell) D_{\iota,\nu}^{\ell}(\bx)\overline{D_{\iota,\nu}^{\ell}(\balpha)}.\label{Addn}
\end{eqnarray}
This is sometimes called the addition formula for Wigner D-functions. 

The expression of the kernel as a series of Chebyshev polynomials \eqref{Cheb}
as well as the addition formula \eqref{Addn} are of prime importance in the study of kernels on $SO(3)$. The first is used to understand the kernel as the fundamental solution of a simple differential operator, which is the goal of the next section, while the second is key to localizing the kernel, which is tackled in Section 4.


%

%
%
\subsection{Smoothness Spaces} \label{ss_smoothness_spaces} Because we are interested in approximating in the $L_p\bigl(SO(3)\bigr)$ metric, we must introduce $L_p$ smoothness spaces on $SO(3)$. This is done in two stages. In the first stage, we introduce the (simpler) Sobolev and H{\" o}lder spaces, whose definitions are accessible, if not familiar to most readers. The second stage is postponed until Section 5, where we introduce Besov spaces. 

Sobolev spaces on manifolds have a variety of equivalent characterizations. Perhaps the most straightforward of these is to import the definition from Euclidean space via a partition of unity and a corresponding set of diffeomorphisms. This is the approach taken in \cite[1.11]{Trieb}.
The Sobolev space on a region $\mathcal{O}$ of $\reals^3$ is defined to be the space of 
functions $f\in L_p(\mathcal{O})$ such that the quantity
\begin{equation}\label{Euc_Sob} 
\|f\|_{W_p^k( \mathcal{O})}^p
: = 
\sum_{\beta \le k} \int_{ \mathcal{O}} |D^{\beta}f(x)|^p \dif x 
\end{equation}
is finite. To transport this definition, let $(\psi_j)_{j=1}^N$ be a partition of unity of $SO(3)$, and let $(\Omega_j,h_j)_{j=1}^N$ be a corresponding collection
of charts (each $\Omega_j$ is an open set in $SO(3)$ containing the closure of $\supp {\psi_j}$ and each $h_j:\Omega_j \to \mathcal{O}_j\subset\reals^3$ is a diffeomorphism). 
\begin{definition}\label{Sobolev}
For $1\leq p<\infty$ the Sobolev space $W_p^k\bigl(SO(3)\bigr)$ consists of functions $f\in L_p$ such that
$$     \|f\|_{W_p^k}^p            
:=\sum_{j=1}^N \|\bigl(\psi_j\circ (h_j^{-1})\bigr)\bigl(f\circ (h_j^{-1})\bigr)\|_{W_p^k(\mathcal{O}_j)}^p$$
is finite.
\end{definition}
In this way we transport the Sobolev space definition from Euclidean space. For smoothness measured in $L_{\infty}$, although the definition makes sense, we end up using the $C^k$ spaces, whose norm is defined according to the same procedure.

%
%
%
An important observation, of use in future sections, is the mapping property of the Laplace \!--\! Beltrami operator. For $j\le k/2$,
 $$\|\Delta^j f \|_p \lesssim \|f\|_{W_p^k}. $$
 This follows from the well known expression of $\Delta$ in local coordinates \cite[7.2.5]{Trieb}: namely
$(\Delta f) \circ h = L (f\circ h)$ where  
$$L = \frac{1}{\sqrt{\det g}}\sum_{j,k=1}^d\frac{\partial}{\partial x_j}\left( \sqrt{\det g} g^{jk} 
\frac{\partial}{\partial x_k}\right)$$ 
is a second order differential
operator having smooth coefficients (here $g$ is an invertible matrix having
smooth entries and $g^{jk}$ is the $j,k$ entry of its inverse).

Another useful observation concerns the density of $C^{\infty}$ functions in 
$W_p^k(SO(3))$, (and in $C^{k}(SO(3))$ when $p=\infty$), which follows from the density
of such functions in $\reals^3$. 
Each real valued function $\bigl(\psi_j f\bigr)\circ (h_j^{-1})$ can be approximated with error $\varepsilon$ in $W_p^k$ by a function
$g_j:\mathcal{O}_j\to \reals$ (via mollification, for instance). 
The desired smooth function is 
$$ 
G := \sum_{j=1}^N \tilde{\psi}_j  \times G_j 
:=  \sum_{j=1}^N\tilde{\psi}_j  \times\bigl( g_j\circ h_j\bigr).$$ 
Here we utilize a 
second family of smooth functions: each
$\tilde{\psi}_j: SO(3)\to \reals$ is a $C^{\infty}$  cut-off function satisfying 
$\supp {\tilde{\psi}_j} \subset \Omega_j$
 and $\tilde{\psi}_j(\balpha) = 1$ for $\balpha \in
\supp{\psi_j}$ 
(the function $\tilde{\psi}_j  \times G_j$ is extended by $0$ outside of
$\Omega_j$). It follows that 
$G - f = \sum_{j=1}^N  \tilde{\psi}_j (G_j - \psi_j f )$
and
\begin{multline*}
\|G - f\|_{W_p^k(SO(3))}\\
 \le
 \sum_{\ell=1}^N \sum_{j=1}^N
\left\|
  \left[
    \psi_{\ell} 
    \times 
    \tilde{\psi}_j
    \times 
    \left\{g_j - (\psi_j\times f) \circ (h_j^{-1})\right\}
     \circ h_j 
  \right]
    \circ h_{\ell}^{-1}   
\right\|_{W_p^k(\mathcal{O}_{\ell})}\\
\le
\sum_{\ell,j=1}^N 
C_{\ell,j}
  \left\|    \left(\psi_{\ell} \times \tilde{\psi}_j \right)\circ h_{\ell}^{-1}  \right\|_{C^k(\Omega_{\ell})}
  \left\|g_j - (\psi_j\times f) \circ (h_j^{-1})\right\|_{W_p^k(\mathcal{O}_j)}
\end{multline*}
where the last line follows from the fact that each $h_j\circ h_{\ell}^{-1}$ is a diffeomorphism
from $h_{\ell}(\Omega_j\cap\Omega_{\ell})$ to $h_{j}(\Omega_j\cap\Omega_{\ell})$.
Therefore, $\|G - f\|_{W_p^k(SO(3)} \le C \varepsilon$, where
$C$ depends only on the two families of smooth functions 
$(\psi_j)_{j=1}^N$ and $(\tilde{\psi}_j)_{j=1}^N$, and on the diffeomorphisms $h_j$. 

%
%
%
\section{Surface Splines on $SO(3)$}

In this article we investigate approximation properties of the {\em surface splines} (perhaps more accurately the `rotational surface splines'), which are the functions derived from $\phi_s(t) := t^{2s}$ or, alternatively $\psi_s (t)=(1-t^2)^s$ with the index, $s$, a positive, pure half integer (which gives the order of zero of the kernel or its smoothness) related to the `order' of the kernel by $s = s(m) = \frac{2m-3}{2}$. We have,
\begin{eqnarray}\label{surface_spline}
k_m(\bx,\balpha) &:=& 
               \left(\sin \left(\frac{\omega(\balpha^{-1}\bx)}{2}\right)\right)^{2m-3} = \phi_{s}\left( \sin \left(\frac{\omega(\balpha^{-1}\bx)}{2}\right)\right)\\
                &=& 
                  \left(1- \cos^2 \left(\frac{\omega(\balpha^{-1}\bx)}{2}\right)\right)^{\frac{2m-3}{2}}=\psi_{s}\left(\cos\left(\frac{\omega(\balpha^{-1}\bx)}{2}\right)\right).\nonumber
\end{eqnarray}
To ensure continuity of the kernel, we will assume throughout the article that $m\ge 2$.

An analysis of their Wigner D-function expansions will reveal two remarkable qualities.
First, the kernels are conditionally positive definite; this property and its consequences
are discussed in \ref{ss_cpd}. 
Second, they satisfy integral identities, valid for $f\in W_p^{2m}(SO(3))$, 
of the form
\begin{equation}\label{rep}
f = \int_{SO(3)} \L_m f(\balpha) k_m(\cdot,\zb  \alpha) \dif\mu(\balpha) 
\end{equation}
where $\L_m$ is an elliptic differential operator of order $2m$ on $SO(3)$. 
To be sure, there are many kernels satisfying a reproduction
formula like (\ref{rep}) and many of them are conditionally positive definite. 
However, in general, it is unclear that they will have a `nice,' closed form expression 
similar to (\ref{surface_spline}).
Similarly, there is no guarantee that they will have the property of localizability 
(developed in the next section),
which is key to our understanding of the approximation power of the kernel (as discussed
in Section \ref{s_main}).

We organize the section as follows. In \ref{ss_cheb}, we expand (\ref{surface_spline})
in terms of Chebyshev polynomials. This expansion is used to show that the 
kernels are conditionally positive definite, which leads to a discussion of some
practical approximation schemes. This is treated in \ref{ss_cpd}. In \ref{ss_polyharmonic}
we establish the reproduction formula (\ref{rep}) and demonstrate that the 
underlying differential operator $\L_m$ is of the form $p(\Delta)$, a degree $m$ polynomial in the Laplace--Beltrami operator.
\subsection{Chebyshev coefficients of the surface splines}\label{ss_cheb}
Our investigation of $k_m$ begins with studying its decomposition in characters (equivalently, the decomposition of $\psi_s$ in even Chebyshev polynomials).
\begin{lemma}\label{ss_coeff}
For $m\ge2$,  
the $\ell^{\mathrm{th}}$ coefficient of $k_m$ in its expansion in terms of characters, cf. \eqref{Cheb}, is
\begin{equation}
\widetilde{k}_m(\ell) =  \frac{2}{\pi} \frac{(2m-2)!}{(-4)^{m-1}}\,\prod_{j=-(m-1)}^{m-1} [\ell +j+\tfrac12]^{-1}.
\end{equation}
\end{lemma}
\begin{proof}
We consider $k_m(\bx,\balpha)=k_m(\zb\alpha^{-1}\bx)=\sum_{\ell = 0}^{\infty} \widetilde{k}_m(\ell) \mathfrak{c}_\ell(\zb\alpha^{-1}\bx)$ and we want to determine the coefficients $\widetilde{k}_m(\ell)$.  Utilizing the fact that the characters
are orthonormal we wish to compute 
$$
\widetilde{k}_m(\ell) = \int_{SO(3)}k_m(\bx)\overline{\mathfrak{c}_\ell(\bx)}\dif\mu(\bx).
$$
By a change of variable, we get 
$$
\widetilde{k}_m(\ell)=\frac{2}{\pi} \int_{0}^{\pi}\left(\sin\left(\frac{\omega}{2}\right)\right)^{2m-3}  \mathcal{U}_{2\ell}\left(\cos \left(\frac{\omega}{2}\right)\right)  \sin^2\left(\frac{\omega}{2}\right) \dif \omega.$$
Rewriting $ \mathcal{U}_{2\ell}\left(\cos \left(\frac{\omega}{2}\right)\right)$ in the familiar form
$$\mathcal{U}_{2\ell}\left(\cos \left(\frac{\omega}{2}\right)\right) = 
\frac{\sin \left((2\ell+1)\frac{\omega}{2}\right)}{\sin \left(\frac{\omega}{2}\right)}$$
and applying another change of variable $\theta = \omega/2$, we obtain:
$$\widetilde{k}_m(\ell) = \frac{4}{\pi} \int_{0}^{\pi/2}(\sin\theta)^{2m-2} \, \sin \bigl((2\ell+1)\theta \bigr)  \dif \theta.$$
Such integrals are easily attacked using simple trigonometric identities. They can also be reduced to the observation that repeatedly  multiplying a periodic function by $\sin \theta$ is equivalent to applying a difference operator to its Fourier coefficients.
Doing so, we obtain
\begin{equation*}\label{trig_id}
(\sin\theta)^{2M} \, \sin \bigl(L\theta \bigr) =\\
 \left(-\frac{1}{4}\right)^{M}
\sum_{j=0}^{2M}  (-1)^{j}\; {{2M}\choose j} \sin\bigl( [L -(2M-2j)]\theta\bigr).
\end{equation*}
Applying this identity 
with $M=m-1$ and $L=2\ell+1$, and integrating each term
$\sin\bigl( [2\ell+3 -(2m-2j)]\theta\bigr)$ over the interval $[0,\frac{\pi}{2}]$,
 we see that 
\begin{eqnarray*}
\widetilde{k}_m(\ell) 
&=& 
\left(-\frac{1}{4}\right)^{m-1} \frac{4}{\pi}
\sum_{j=0}^{2m-2} 
  \frac{ (-1)^{j}}{ 2\ell+3 - (2m-2j) }
  {{2m-2}\choose j} \\
 &=&\left(-\frac{1}{4}\right)^{m-1} \frac{2}{\pi}
\sum_{j=0}^{2m-2} 
  \frac{ (-1)^{j}}{ \ell-m+\tfrac{3}{2} +j }
  {{2m-2}\choose j} .
\end{eqnarray*}
Hence, $\widetilde{k}_m(\ell)$ is a difference operator applied to the rational function
$x\mapsto x^{-1}$. We simplify this type of expression in the elementary formula:
\begin{equation}\label{diff}
\sum_{j=0}^M (-1)^j {M\choose j}\frac{1}{L+j} 
= 
\frac{ M!}{(L)(L+1)\dots(L+M-1)(L+M)}.
\end{equation}
Thus, we can simplify $\widetilde{k}_m(\ell)$, by using (\ref{diff}) with $L=\ell-m+\tfrac{3}{2}$ and $M=2m-2$:
$$\widetilde{k}_m(\ell) =
 \frac{2}{\pi}\left(-\frac{1}{4}\right)^{m-1}\frac{(2m-2)!}{(\ell -m+3/2)\dots(\ell+m-1/2)},$$
which proves the lemma. 
\end{proof}
%
%
\subsection{Treating data with surface splines}\label{ss_cpd}
At this point we observe that the kernels $k_m$ are each conditionally positive definite of
order $\ell_0\ge m-2$, and, hence, well suited for solving interpolation problems.
Being conditionally positive definite of order $\ell_0$ means that for any $\zb \Xi \subset SO(3)$, 
 $\sum_{\zb \zeta\in\zb\Xi} \sum_{\zb \xi\in\zb\Xi} \zb a_{\bxi} \overline{\zb a_{\zb \zeta}}\zb A_{\bxi,\zb \zeta} > 0$ for any $\zb a\neq \zb0$ in the subspace defined by $\zb B \zb a = \zb 0$, where
\begin{equation}\label{E_collocation}
\zb A := [k_m(\bxi,\zb \zeta)]_{\bxi,\zb \zeta \in \zb \Xi}
\quad\text{and}\quad
\zb B := [D^{\ell}_{\iota,\nu}(\bxi)]_{\ell=0,\ldots,\ell_0,-\ell\leq\iota,\nu\leq\ell;\bxi \in \zb \Xi}.
\end{equation}

Indeed, conditional positive definiteness follows directly from 
the fact that
$\widetilde{k}_m(\ell)$ does not change sign for $\ell\ge m-2$
 (see \cite[Proposition 4.3]{DNW}).

We now present three classical, practical approximation schemes involving the kernel $k_m$. Namely, we approximate by functions of the form
\begin{equation}\label{interpolant}
s_{\zb\Xi}=\sum_{\bxi\in\zb \Xi}\alpha_{\bxi} k_m(\cdot, \bxi)+\sum_{\ell=0}^{\ell_0}\sum_{\iota,\nu=-\ell}^\ell \beta_{\iota,\nu}^\ell D^\ell_{\iota,\nu},
\end{equation}
where the kernel coefficients satisfy
\begin{equation}\label{aux}
\sum_{\bxi \in\zb \Xi}\alpha_{\bxi}D^{\ell}_{\iota,\nu}(\bxi)=0,\quad 0\leq\ell\leq\ell_0;-\ell\leq\iota,\nu\leq\ell.
\end{equation}

\noindent{\bf Interpolation} We define the interpolant to a function $f$ at the centers 
$\zb\Xi$ as the function of the form (\ref{interpolant}) 
where the interpolation conditions
$$
s_{f,\zb\Xi}(\bxi)=f(\bxi),\quad \bxi\in\zb\Xi,
$$
and auxiliary conditions (\ref{aux}) are satisfied.

The interpolant is determined by solving the system
\begin{equation}\label{intsystem}
\left(
\begin{array}{cc}
\zb A& \zb B^T\\ \zb B& \zb 0
\end{array}
\right)
\left(
\begin{array}{c}
\zb \alpha\\
\zb \beta
\end{array}
\right)=
\left(
\begin{array}{c}
\bf f\\
\zb 0
\end{array}
\right),
\end{equation} where  
${\bf f}=(f(\bxi))_{\bxi \in\zb \Xi}$. 
It is well-known that system \eqref{intsystem} is uniquely solvable for any set of centers $\zb\Xi$ that is unisolvent for the polynomial space $\zb\Pi_{\ell_0}$, see \cite[Chapter 8.5]{Wend} (we provide sufficient conditions for such sets, in terms of the fill distance,  in the next section). Moreover, for our kernels $k_m$, setting up this system is particularly easy (see Note 
\ref{euler_angle_collocation} 
below).
We refer the interested reader to \cite{Wend, DNW} for  background on interpolation by conditionally positive definite kernels.

\noindent{\bf Tikhonov regularization} For treating noisy data, a common approach is to find the kernel approximant
minimizing a certain quadratic form. For noisy data $( y_{\bxi})_{\bxi\in \zb \Xi}$ sampled at 
 $\zb \Xi\subset SO(3)$, 
we wish to solve the problem
$$
\min_{s \in W_2^m(SO(3))} \left(\sum_{\bxi \in \zb \Xi} |s(\bxi) -  y_{\bxi}|^2 + \lambda |s|_{\mathcal{N}}^2\right) 
$$
where we use the seminorm 
$|s|_{\mathcal{N}}^2 = \sum_{\ell\ge \ell_0}\sum_{-\ell\le \iota,\nu\le \ell} |\hat{s}_{\iota,\nu}^{\ell}|^2|\widetilde{k}_m(\ell)|^{-1}$
and a smoothing parameter $\lambda>0$ is chosen based on the level of noise inherent
in the data.
The minimizer is of the form (\ref{interpolant}).
Because $k_m$ is conditionally positive definite, 
the minimizer is determined by solving the system of equations 
\begin{equation}\label{intsystem1}
\left(
\begin{array}{cc}
\zb A + \lambda \mathrm{Id}& \zb B^T\\ \zb B& \zb 0
\end{array}
\right)
\left(
\begin{array}{c}
\zb \alpha\\
\zb \beta
\end{array}
\right)=
\left(
\begin{array}{c}
\bf y\\
\zb 0
\end{array}
\right).
\end{equation}
 See \cite[2.4]{Wahb} for a detailed discussion of this problem.

\noindent{\bf Least squares approximation} For a unisolvent set $\zb \Xi$, the space 
$$S_{\zb \Xi}:=\left\{ s=\sum_{\bxi\in\zb \Xi}\alpha_{\bxi} k_m(\cdot, \bxi)+ p
\mid p\in \zb\Pi_{\ell_0},\  (\alpha_{\bxi})_{\bxi\in\zb\Xi} \ \,\text{satisfying} \ (\ref{aux})\right\}$$
has dimension $\#\zb\Xi$, and one can construct the least squares operator
$P_{\zb \Xi}$, which is the orthogonal projector $P_{\zb \Xi}:L_2(SO(3)) \to S_{\zb \Xi}$
with range $S_{\zb \Xi}$. 

One may calculate the least squares projector by using any basis $(v_{\bxi})_{\bxi}$ for 
$S_{\zb \Xi}$ (say the basis of Lagrange functions $\chi_{\bxi}$ where $\chi_{\bxi}(\bzeta) = \delta_{\bxi,\bzeta}$). 
In this case, 
the projector is simply 
$P_{\zb \Xi} f = \sum_{\bxi\in \zb\Xi} a_{\bxi} v_{\bxi}$, with coefficients $\zb a = (a_{\bxi})_{\bxi\in\zb \Xi}$
given by  $ \zb a =\zb G^{-1} \zb f$, where $\zb G_{\bxi,\bzeta} = \langle v_{\bxi},v_{\bzeta}\rangle$ is
the Gram matrix for the basis and
and $\zb f = \bigl(\langle v_{\bxi}, f\rangle\bigr)_{\bxi\in\zb \Xi}.$ Since the elements of
$S_{\zb \Xi}$ are continuous and $SO(3)$ is compact, the projector extends
as an operator to each space $L_p(SO(3)),$ $1\le p\le \infty$. The $L_p$ stability of such operators on compact Riemannian manifolds has recently been investigated in
\cite{HNSW}; an open problem, \cite[Conjecture 3]{HNSW}, is whether the stability and approximation results of 
that article hold for the kernels presented here.

\begin{note}\label{euler_angle_collocation}
Suppose we are given centers $\zb\Xi=\{\bxi_j\in SO(3),j=0,\ldots,M-1\}$ 
in terms of Euler angles (as described in Section \ref{ss_basic_facts}), i.e., each $\bxi_j$ is determined by $(\varphi_{1,j},\theta_j,\varphi_{2,j})$.
It is easy to show that we can write the distance of two centers $\bxi_i$ and $\bxi_j$ in terms of Euler angles; in fact we have
\begin{eqnarray*}
\cos\left(\frac{\dist(\bxi_i,\bxi_j)}{2}\right)\!\!\!&=&\!\!\!{\Bigg{(}}\Bigg|\cos\left(\frac{\varphi_{1,i}-\varphi_{1,j}}{2}\right)\cos\left(\frac{\varphi_{2,i}-\varphi_{2,j}}{2}\right)\cos\left(\frac{\theta_i-\theta_j}{2}\right) \\ 
&&\!\!\!\!\!\!-\sin\left(\frac{\varphi_{1,i}-\varphi_{1,j}}{2}\right)\sin\left(\frac{\varphi_{2,i}-\varphi_{2,j}}{2}\right)\cos\left(\frac{\theta_i+\theta_j}{2}\right)\Bigg|\Bigg).
\end{eqnarray*}
Thus, we can directly write down the closed form expression for the entry $\zb A_{i,j}$ of the matrix $\zb A$ as
$$
\zb A_{i,j}=\left(1-\cos^2\left(\frac{\dist(\bxi_i,\bxi_j)}{2}\right)\right)^\frac{2m-3}{2}.
$$
Furthermore, the nonequispaced Fourier matrix $\zb B$ can be computed
directly from (\ref{explicit_form_matrix_elements_SO(3)}).
Such matrices have been studied in great detail and can be handled efficiently, see \cite{PoPrVo09}. 

On the other hand, if the centers are given via their rotation axis and rotation angle, we can immediately compute the corresponding Euler angles, 
using the direct relation between the Euler angle and the axis-angle parameterization of $SO(3)$, see \cite[Chapter 1.4.4]{Varsh}).
By doing so, we can again easily set up the matrices  $\zb A$ and $\zb B$ as shown above.
\end{note}

\subsection{Integral representation}\label{ss_polyharmonic}
Returning to our investigation of the coefficients $\widetilde{k}_m(\ell)$, we observe that 
$\widetilde{k}_m(\ell)$ can be simplified by arranging factors symmetrically around $\ell+1/2$. Doing this, we have: 
\begin{eqnarray}
\widetilde{k}_m(\ell) 
&=&
\frac{2}{\pi}\frac{(2m-2)!}{(-4)^{m-1}}  
\left[\prod_{j=1}^{m-1} \frac{ 1}{ (\ell+1/2) - j } \right]\frac{1}{\ell+1/2} \left[\prod_{j=1}^{m-1} \frac{ 1}{ (\ell+1/2)  + j }\right] \nonumber \\
&=& 
\frac{2}{\pi}\frac{(2m-2)!}{(-4)^{m-1}} 
\frac{ 1 }{(\ell+1/2)}\prod_{j=1}^{m-1} \frac{ 1}{ (\ell+1/2)^2 - j^2 }.
\end{eqnarray}

The last line comes from rearranging terms, and noting that $[J - k ][J+k] = J^2 -k^2$. 
Since $(\ell+1/2)^2= \ell(\ell+1) +1/4$, we can rewrite the expression in terms of $\ell(\ell+1)$  and obtain (after reindexing)
\begin{eqnarray}
\widetilde{k}_m(\ell) 
&=&\frac{1}{\pi}\frac{(2m-2)!}{(-4)^{m-1}}(2\ell+1)  \prod_{j=0}^{m-1} \frac{ 1}{ \ell(\ell +1) - [j^2-1/4] }.
\end{eqnarray}

The coefficients $\widetilde{k}_m(\ell)$ involve the {\it symbol} for $\Delta$, $\sigma(\Delta) = \ell(\ell+1)$. In fact, these are the reciprocals of a polynomial in the symbol, $p\bigl(\sigma(\Delta)\bigr)^{-1}$, which is nonvanishing, because the numbers $[j^2-1/4]$ never coincide with eigenvalues of $\Delta$.
So the operator induced by $k_m$ inverts the elliptic operator $p(\Delta)$. This is the point of the following lemma:
%
%
\begin{lemma} \label{ss_is_Green}  The kernel 
$(\bx,\balpha)\mapsto k_m(\bx,\balpha)$ satisfies the formula
$$f = \int_{SO(3)} k_m(\cdot,\zb  \alpha) \,\L_{m} f(\balpha)\,\dif \mu(\balpha) $$
using the operator of order $2m$,
$\L_{m}:=\frac{\pi(-4)^{m-1}}{(2m-2)!}\prod_{j=0}^{m-1} (\Delta - r_j)$ with
$r_j:=j^2-1/4$, and which holds for all $f\in W_p^{2m}\bigl(SO(3)\bigr)$ (or $C^{2m}(SO(3))$ when $p=\infty$).
\end{lemma}
\begin{proof}
For a sufficiently smooth $f$ (say $C^{\infty}$ initially), we have the absolutely convergent series $f = \sum_\ell \sum_{\iota,\nu} c_{\iota,\nu}^{\ell} D_{\iota,\nu}^{\ell}$. 
Using the addition formula \eqref{Addn}, property \eqref{LaBe} and the orthogonality of the Wigner D-functions we get
\begin{eqnarray*}
&\mbox{}&\int_{SO(3)} k_m(\bx,\zb  \alpha) \L_{m} f (\balpha)\,\dif \mu(\balpha)  \\
&=&\int_{SO(3)} \sum_{\ell}\widetilde{k}_m(\ell) {\mathfrak c}_{\ell}(\balpha^{-1}\bx) \times\\
&&  \quad\quad \frac{\pi(-4)^{m-1}}{(2m-2)!}\prod_{j=0}^{m-1} \big( \Delta - r_{j}\big) \sum_{\ell'}\sum_{\iota',\nu'} c_{ \iota',\nu'}^{\ell'}D_{\iota',\nu'}^{\ell'}(\balpha) \dif \mu(\balpha)\\
&=&\int_{SO(3)} \sum_{\ell}\sum_{\iota,\nu}\widetilde{k}_m(\ell)D_{\iota,\nu}^{\ell}(\bx)\overline{D_{\iota,\nu}^{\ell}(\balpha)} \times\\
&&  \quad\quad\sum_{\ell'}\sum_{\iota',\nu'} c_{ \iota',\nu'}^{\ell'} \frac{\pi(-4)^{m-1}}{(2m-2)!}\prod_{j=0}^{m-1} \big( \ell'(\ell'+1) - r_{j}\big) D_{\iota',\nu'}^{\ell'}(\balpha) \dif \mu(\balpha)\\
&=&\sum_\ell\sum_{\iota,\nu}\widetilde{k}_m(\ell) \frac{\pi(-4)^{m-1}}{(2m-2)!}\prod_{j=0}^{m-1} \left[ \ell(\ell+1) - r_{j}\right] \frac{1}{2\ell+1} c_{\iota,\nu}^{\ell} D_{\iota,\nu}^{\ell}(\bx)=f(\bx).
\end{eqnarray*} 
The reproducing formula follows for all $f$ by a limiting argument, using the density of $C^{\infty}$ in $W_p^{2m}$ for $1\le p<\infty$ (or $C^{2m}$ when $p=\infty$) and the boundedness of $\L_m$. 
\end{proof}

\section{Localizing the Kernel}
For a set of centers $\zb\Xi\subset SO(3) $ we now wish to investigate a `coefficient kernel' 
$\a:\zb\Xi\times SO(3)\to \reals$
that will allow us to easily replace $k_m(\bx,\balpha)$ with  $\bar{k}(\bx,\balpha):=\sum_{\bxi \in \zb \Xi} \a(\bxi,\balpha) k_m(\bx,\bxi)$
in the representation (\ref{rep}). To
do so, we must estimate the cost of replacing the kernel, $e_{k_m}$, given by the `error kernel':
$$e_{k_m}(\bx,\balpha):=|k_m(\bx,\balpha) - \sum_{\bxi \in\zb  \Xi} \a(\bxi,\balpha) k_m(\bx,\bxi)|.$$ 
This is a bounded, rapidly decaying function, 
inducing an operator on $L_p$ with small norm. 

The outline of this section is as follows. In \ref{ss_CKC} we develop conditions on 
the coefficient kernel $\a(\bxi,\balpha)$ necessary for proving our results, and it
is shown that such conditions are satisfied for sufficiently dense centers. 
A practical procedure for producing the coefficients is given in \ref{ss_practical}.
In \ref{ss_exchange}, we estimate the decay of the error kernel $e_{k_m}(\bx,\balpha)$.

\subsection{Coefficient kernel conditions}\label{ss_CKC}
The two key quantities we need to resolve are the polynomial {\em precision} 
(the degree of Wigner
D-functions reproduced by the coefficients $\a(\cdot,\balpha)$) and the
rate of decay of the error kernel. 
As in the Euclidean and spherical setting, these are related: the greater the polynomial precision, the more local the replacement error. 

\begin{definition}[Coefficient Kernel Conditions]\label{CKC}
For a set of centers $\zb\Xi \subset SO(3)$ there is a number $\rho = \rho(\zb \Xi)$, reflecting the density of $\zb \Xi$ in $SO(3)$, and a measurable kernel $\a:\zb \Xi\times SO(3) \to \reals$ satisfying the
following three conditions:
\begin{description}
\item[\bf CKC 1 (Support)] $\a(\bxi,\balpha) = 0$ for ${\rm dist}(\balpha,\bxi)>\rho$.
\item[\bf CKC 2 (Precision)] For any polynomial of degree at most $L$, i.e., for any $p \in \zb \Pi_{L},$ the following holds:
\begin{equation*}
\sum_{\bxi\in \zb \Xi} \a(\bxi,\balpha) p(\bxi) = p(\balpha) \quad \text{for every $\balpha\in SO(3)$.}
\end{equation*}
\item[\bf CKC 3 (Stability)] There is a constant $K\ge 1$ such that for every $\balpha\in SO(3)$, we have 
\begin{equation*}
\sum_{\bxi \in \zb \Xi} |\a(\bxi,\balpha)|\le K.
\end{equation*}
\end{description}
\vspace{.1in}
\end{definition}
As in the Euclidean case, such a local polynomial reproduction property holds for sufficiently dense centers. 
In order to show this we need the following fact about polynomials on $SO(3)$.
\begin{proposition}\label{Wigner_to_trig}
Let $p\in\zb\Pi_L$. Then $p$ restricted to a geodesic segment on $SO(3)$ gives a trigonometric polynomial of degree at most $L$.
\end{proposition}
\begin{proof}
Let $p=\sum\limits_{\ell=0}^L\sum\limits_{\iota,\nu=-\ell}^\ell a_{\iota,\nu}^\ell D_{\iota,\nu}^\ell\in \zb\Pi_L$ and an arbitrary geodesic segment $G=\{\bx_0\zb s_{\eta}(\omega):\omega\in[0,t]\}$ of length $t\in[0,2\pi)$ on $SO(3)$ be given. 
Here $\zb s_{\eta}(\omega)$ denotes a rotation about the rotation axis $\eta\in\mathbb{S}^2$ with rotation angle $\omega\in[0,2\pi)$.
Then we can find $\zb y\in SO(3)$ such that $\zb s_{\eta}(\omega)=\zb y^{-1}\zb s_z(\omega)\zb y$ for all $\omega\in [0,t]$.
The rotation $\zb s_z(\omega)$ can, for example, be represented by the Euler angles $(\omega,0,0)$.
Thus, we get, by using the representation \eqref{explicit_form_matrix_elements_SO(3)}, for $\bx(\omega)=\bx_0\zb s_{\eta}(\omega)\in G$ 
\begin{eqnarray*}
p(\bx(\omega))&=&p(\bx_0\zb y^{-1}\zb s_z(\omega)\zb y)=\sum\limits_{\ell=0}^L\sum\limits_{\iota,\nu=-\ell}^\ell a_{\iota,\nu}^\ell D_{\iota,\nu}^\ell(\bx_0\zb y^{-1}\zb s_z(\omega)\zb y)\\
&=&\sum_{\ell=0}^{L}{\rm Tr}(\hat{A}^T_\ell D^\ell(\bx_0\zb y^{-1}\zb s_z(\omega)\zb y)))\\
&=&\sum_{\ell=0}^{L}{\rm Tr}(\underbrace{D^\ell(\zb y)\hat{A}^T_\ell D^\ell(\bx_0\zb y^{-1})}_{\hat{B}_\ell^T}D^\ell(\zb s_z(\omega))\\
&=&\sum\limits_{\ell=0}^L\sum\limits_{\iota,\nu=-\ell}^\ell {b}_{\iota,\nu}^\ell D_{\iota,\nu}^\ell(\zb  s_z(\omega))=\sum_{\ell=0}^{L}\sum_{\iota=-\ell}^\ell b_{\iota,\iota}^\ell e^{-i\iota\omega}.
\end{eqnarray*}
We note that the operator $\hat{B}_\ell$ and its entries ${b}_{\iota,\nu}^\ell$ are independent of $\omega$. The final equality is a consequence of (\ref{explicit_form_matrix_elements_SO(3)}).
\end{proof}
The Coefficient Kernel Conditions are easily satisfied. The following lemma shows that for sufficiently dense centers, one can find a suitable kernel.
\begin{lemma}\label{M_Z}
Given $L$, an integer reflecting the desired precision, and given centers $\zb \Xi$ having 
fill distance $h\le h_0$ (where $h_0 := \frac{\pi}{ C L^2}$, a constant depending only on $L$), 
there exists 
a coefficient kernel $\a:\zb\Xi\times SO(3)\to \reals$ satisfying the CKCs with radius
$\rho = C L^2 h$ (with $C>0$ a global constant) and $K=2$.
\end{lemma}
Before providing the proof, we remark that the restriction $\rho\le \pi$ forces 
$h\le \frac{\pi}{ C L^2},$ which is the only restriction placed on $h$.
\begin{proof}
%
%
%
Constructing a coefficient kernel follows a `norming set argument', which relies on being able to capture the norm of a polynomial over a ball by its restriction to the centers inside the ball:
\begin{equation}\label{restriction}
\|P\|_{L_{\infty}(B(\zb p_0,\rho))}\le 2\|P_{|_{\zb \Xi\cap B(\zb p_0,\rho)}}\|_{\ell_{\infty}(\zb \Xi\cap B(\zb p_0,\rho))}.
\end{equation}
Utilizing this to obtain a local polynomial reproduction is a fairly standard procedure,
and the basics of our argument can be found in  \cite[Lemma 4.2]{Hang} or \cite[Chapter 3]{Wend}. So all we need to show is \eqref{restriction}.

The inequality (\ref{restriction}) relies on locating a center 
$\zb \xi\in \zb \Xi$ near the point where the uniform norm is attained,  
$\zb p\in B (\zb p_0,\rho)$, 
so that a geodesic segment of length $\omega$, 
starting at $\zb p$ and passing through $\zb \xi$ 
lies entirely within $B(\zb p_0,\rho)$. 
Specifically, we require
\begin{equation}\label{cone}
\frac{\dist(\zb p,\zb \xi)}{\omega }\le \frac{1}{16 L^2}.
\end{equation}
The geodesic segment is determined by the curve $\gamma:[-\omega/2,\omega/2]\mapsto SO(3)$,
parameterized by arclength.
The main estimate we employ is Videnski\u\i's inequality \cite{Vid} which reads as follows. Let $T$ be a trigonometric polynomial of degree at most $L$, then
$$|T'(t)|\le 2 L^2 \cot (\omega/4) \|T\|_{L_{\infty}(-\omega/2,\omega/2)}\qquad
\text{for $\omega/2<\pi$, $|t|\le \omega/2.$}$$
The trigonometric polynomial $\tau = P\circ \gamma$ is the restriction of $P$, the polynomial in question, to the geodesic segment. 
We then observe that 
$\tau(-\omega/2) - \tau(-\omega/2+\dist(\zb p,\zb \xi)) \le \int_{-\omega/2}^{ -\omega/2 +\dist(\zb p,\zb \xi)} |\tau'(t)|\dif t$. 
Applying Videnski\u\i's inequality 
$$P(\zb p) - P(\zb \xi) 
\le  
\frac{ 8 L^2 }{\omega}  \dist(\zb p,\zb \xi)\|\tau\|_{L_{\infty}(-\omega/2,\omega/2)}
  \le \frac{1}{2} P(\zb p).
$$
This gives us the desired inequality: $P(\zb p) \le 2P(\zb \xi)$.

The fact that (\ref{cone}) holds follows by employing a cone condition. On $SO(3)$, a cone $c(\zb q,\theta,\nu,r)$ is the union of geodesic segments of length $r$ emanating from $\zb q$ and having the initial tangent vector within $\theta$ of the tangent vector $\nu$. We call $\theta$ the aperture and $\nu$ the direction of the cone. It is an elementary fact that every ball in $SO(3)$ satisfies a uniform cone condition. I.e., there is $\theta>0$ so that for any $\zb p_0\in SO(3)$ and $\rho>0$ and for every point in the ball $B(\zb p_0,\rho)$,
the cone $c(\zb p, \theta,\nu, \rho/2)\subset  B(\zb p_0,\rho)$, where $\nu$ is the tangent vector at $\zb p$ of the
geodesic segment connecting $\zb p$ and $\zb p_0$.


An appeal to geometry shows that there is a global constant $C'$ depending only on $\theta$ and the geometry of $SO(3)$
so that a smaller ball $b\subset SO(3)$ having small radius $h$ can be placed entirely within the cone 
$c(\zb p, \theta,\nu, \rho/2)$ with its center a distance of $C'h$ from $\zb p$. 
Thus there is a center $\zb \xi$ that is a distance
$Ch \leq (C'+1)h$ from $\zb p$, for which the geodesic segment connecting $\zb p$ and $\zb \xi$  of length $\rho/2$ is entirely within $B(\zb p_0,\rho)$. Therefore,
$\rho = 32 L^2 C h$ suffices. 
\end{proof}

\subsection{Constructing coefficients in practice}\label{ss_practical}
Although the previous lemma guarantees the existence of a coefficient 
kernel sufficient for our purposes, it does not provide a
viable method for generating such kernels in practice. 
To do so, we make the extra assumption that the centers are quasiuniform:
 that $h$ and $q$, as defined
in Section 2, are related by $h/q\le \mathfrak{m}$ for a fixed `mesh ratio' $\mathfrak{m}<\infty$.
Armed with this assumption, we can estimate the number of centers in a given ball 
of radius $\rho=CL^2 h$ (as was used in the previous lemma):
$$\#\zb{\Xi}_{\balpha} = \# \bigl(B(\balpha,\rho)\cap \zb{\Xi}\bigr)
\le \frac{b_2(CL^2 h)^3}{b_1 q^3} \le \left(\frac{\pi}{2}\right)^2 C^3 L^6 \mathfrak{m}^3.$$
We may determine a modified coefficient kernel, which we call $\widetilde{\a}$ 
that satisfies the coefficient kernel conditions by taking
the $\ell_2$ minimizing solution at each $\balpha$ of the system
$\zb B \mathsf{a} = {\mathsf{D}_{\balpha}}$,
where 
$({\mathsf{D}_{\balpha}})_{(j,k,\ell)}  = D_{jk}^{\ell}(\balpha)$ is the vector of 
Wigner D-functions evaluated at $\balpha$,
the matrix $\zb B$ is the `non-equispaced Fourier matrix' introduced in (\ref{E_collocation}). I.e., the
matrix 
whose rows are the  D-functions evaluated
at the various $\bxi$ in $\zb{\Xi}_{\balpha}$. I.e., $(\zb B)_{(j,k,\ell),\bxi} = D_{jk}^{\ell}(\bxi)$. 

It follows that  
the $\ell_2$ minimizing solution is given by 
$\zb B^*(\zb B \zb B^*)^{-1} {\mathsf{D}_{\balpha}}$, 
so we obtain a coefficient kernel 
$\widetilde{\a}(\bxi,\balpha)$ 
satisfying conditions CKC 1 and CKC 2 of Definition \ref{CKC}
via
$\bigl(\widetilde{\a}(\bxi,\balpha)\bigr)_{\bxi\in \zb{\Xi}_{\balpha}} =\zb B^* (\zb B \zb B^*)^{-1} {\mathsf{D}_{\balpha}}$ and zero extension. 
Moreover, it is clear that on the subset 
$$\{\balpha' \in SO(3)\mid \zb \Xi_{\balpha'} = \zb\Xi_{\balpha} \}$$
(i.e., the set of $\balpha'$ for which the sets  $\zb \Xi_{\balpha'}$ coincide) each 
$\widetilde{\a}(\bxi,\cdot)$
is a Wigner D-function.  Hence, $\widetilde{\a}$ is measurable.

To show that $\widetilde{\a}$ satisfies CKC 3 we can compare it 
to the kernel $\a$ guaranteed by Lemma \ref{M_Z} to observe that 
\begin{eqnarray*}
\|\widetilde{\a}(\cdot,\balpha)\|_{\ell_1(\zb\Xi)}
&\le &
\left(\#\zb{\Xi}_{\balpha}\right)^{1/2} \|\widetilde{\a}(\cdot,\balpha)\|_{\ell_2(\zb\Xi)}
\le 
  \sqrt{(\pi/2)^2C^{3} L^6 \mathfrak{m}^{3}} \, \|{\a}(\cdot,\balpha)\|_{\ell_2(\zb\Xi)}\\
&\le&
\frac{\pi}{2}\sqrt{ C^{3} L^6 \mathfrak{m}^{3}  }\,\|{\a}(\cdot,\balpha)\|_{\ell_1(\zb\Xi)}
 \le 
 \pi\sqrt{ C^{3} L^6 \mathfrak{m}^{3} }.
 \end{eqnarray*}
 So CKC3 is satisfied with $K=  \pi\sqrt{ C^{3} L^6 \mathfrak{m}^{3} }$.

\subsection{Estimating the cost of replacing the kernels}\label{ss_exchange}
A consequence of the Coefficient Kernel Conditions is that for any $\bx\in SO(3)$ 
and a smooth class function $\kappa$, there is a convenient mechanism to estimate the
error $e_{\kappa}(\bx,\zb\alpha)$ in terms of an interval around $\omega(\balpha^{-1}\bx)$.
For fixed $\bx\in SO(3)$ we choose to express the kernel $\kappa$ in terms of 
$\tau_{\bx}(\bzeta)= \cos(\frac{\omega(\bzeta^{-1}\bx)}{2})$, 
by way of the formula 
$\kappa(\bx,\bzeta) = \vartheta(\tau_{\bx}^2{(\bzeta)})$ 
assuming $\vartheta$  is smooth on the interval 
$\mathcal{I}_{\bx}:=[\min Q_{\bx},\max Q_{\bx}],$ where 
the set $Q_{\bx}$ depends on the rotation angles used, 
$Q_{\bx}:=\{\tau_{\bx}^2{(\balpha)} \}\cup \{\tau_{\bx}^2(\bxi):\bxi\in (\zb \Xi\cap B(\balpha,\rho)\}$.
The cost of replacing the kernel can be estimated in terms of the length of the interval
$\mathcal{I}_{\bx}$ and the size of derivatives of $\vartheta$ 
{\em purely on $\mathcal{I}_{\bx}$}.  
This is the gist of the following lemma:
\begin{lemma}\label{kernel_cost}
Given a kernel  $\kappa(\bx,\bzeta) = \vartheta (\tau_{\bx}^2{(\bzeta)})$, 
with $\vartheta \in C^{L+1}(\mathcal{I}_{\bx})$, and a coefficient kernel satisfying the CKCs with precision $L$, the replacement error 
$e_{\kappa}(\bx,\balpha)\index{\footnote{}}
= 
|\kappa(\omega(\balpha^{-1}\bx)) - \sum_{\bxi \in \zb \Xi} \a(\bxi,\balpha) \kappa(\omega( \bxi^{-1}\bx))|.$ We have the estimate:
\begin{equation}\label{kern_est}
e_{\kappa}(\bx,\balpha)\le 
\frac{\|\a(\cdot,\balpha)\|_{\ell_1(\zb \Xi)}}{(L+1)!}\,\,
  |\mathcal{I}_{\bx}|^{L+1} \, \max_{t\in \mathcal{I}_{\bx}} |\vartheta^{(L+1)}(t)|. 
\end{equation}
\end{lemma}
\begin{proof}
Let both $\bx$ and $\balpha$ be fixed, and choose the Taylor polynomial of degree $L$, $q_{L,t_{\bx}(\balpha)}$, of $\vartheta$ expanded about $t_{\bx}(\balpha) = \tau_{\bx}^2(\balpha)$. 
Now $q_{L,t_{\bx}(\balpha)}$ may be rewritten as a linear combination of even degree Chebyshev polynomials,
$$q_{L,t_{\bx}(\balpha)}(t) =  \sum_{\ell=0}^L c_{\ell}
\mathcal{U}_{2\ell}(\sqrt{t}).$$
 This follows because $\mathcal{U}_{2\ell}$ is of exact degree $2\ell$ and is a linear combination of even degree monomials only, so $\mathcal{U}_{2\ell}(\sqrt{t})$ is a polynomial of exact degree $\ell$, and, thus, the set $\left(\mathcal{U}_{2\ell}(\sqrt{t})\right)_{\ell = 0}^L$ is a basis for set of polynomials up to degree $L$.
Note, furthermore, that $q_{L,t_{\bx}(\balpha)}(t_{\bx}(\balpha))- \sum \a(\bxi,\balpha)q_{L,t_{\bx}(\balpha)}(t_{\bx}(\bxi)) = 0$ by the addition formula, since 
$$
q_{L,t_{\bx}(\balpha)}(\tau_{\bx}^2(\bzeta)) 
= \sum_{\ell=0}^L c_{\ell}\mathcal{U}_{2\ell}(\tau_{\bx}(\bzeta))
=\sum_{\ell\le L}c_{\ell} \sum_{k,m} D_{k,m}^{\ell}(\bx)\overline{D_{k,m}^{\ell}(\bzeta)}
$$
 and each $D_{k,m}^{\ell}(\bzeta)$ is annihilated by 
 $\mu = \delta_{\balpha}-\sum \a(\bxi,\balpha) \delta_{\bxi}$. 

Consequently, 
$
e_{\kappa}(\bx,\balpha) 
\le 
\sum_{\bxi} |\a(\bxi,\balpha)| \:|\vartheta(t_{\bx}(\bxi)) - q_{L,t_{\bx}(\balpha)}(t_{\bx}(\bxi))|,
$ 
and the lemma follows from the remainder in Taylor's theorem: 
$$
|\vartheta(t_{\bx}(\bxi)) - q_{L,t_{\bx}(\balpha)}(t_{\bx}(\bxi))|
\le 
\frac{1}{(L+1)!} |t_{\bx}(\bxi)- t_{\bx}(\balpha)|^{L+1} \max_{t\in \mathcal{I}_{\bx}}|\vartheta^{(L+1)}(t)|.
$$
\end{proof}
We  now  measure $e_{k_m}$, the cost of replacing $ k_m(\bx,\balpha)$ by a linear
combination of its copies, $\sum_{\xi} \a(\bxi,\balpha)k_m(\bx,\bxi)$. 
%
%
\begin{lemma}\label{ss_loc}
Assume $\a$ is a coefficient kernel satisfying the CKCs with radius $\rho$, precision $L\ge 2m$ and stability constant 
$K=\sup_{\balpha}\|\a(\cdot,\balpha)\|_{\ell_1(\zb \Xi)} $. 
Then the replacement error of the kernel $k_m$ satisfies 
\begin{equation*}
e_{k_m}(\bx,\balpha) \le {\const}(L,m)\, K
\rho^{2m-3}\left(1+ \frac{\dist (\bx,\balpha)}{\rho}\right)^{2m-4-L}.
\end{equation*}
where $\const(L,m)$ is a constant that only depends on $L$ and $m$.
\end{lemma}
\begin{proof}
We break this into two parts, the first concerns $\zb x$ near to $\balpha$ (say $\dist(\bx,\balpha) \le 2\rho$) 
where we can make use of the high order zero of the kernel, while the second treats 
$\bx$ sufficiently far from $\balpha$ so that we can utilize the smoothness of the kernel and Lemma \ref{kernel_cost}.

{\bf Case I:} 
Assume $\dist(\bx,\balpha) \le 2\rho$. 
This implies that $\dist(\bx,\bxi) \le 3\rho$ for every $\bxi\in \supp{\a(\cdot,\balpha)}$. 
Writing the kernel as $k_m(\bx,\bzeta) = \left(\sin \frac{\dist(\bx,\bzeta)}{2}\right)^{2m-3}$
and using the stability of the coefficient kernel, we see that
\begin{eqnarray*}
|k_m(\bx,\balpha)  - \sum \a(\bxi,\balpha) k_m(\bx,\bxi)| &\le& (1+K) \left(\sin\frac{3}{2}\rho\right)^{2m-3} 
\\
&\le& (1+K)   \left(\frac{3}{2}\right)^{2m-3}\rho^{2m-3}.  
\end{eqnarray*}
Because $1+K \le 2K \le (3/2)^2 K$, it follows that 
$e_{k_m}(\bx,\balpha)$ is less than  $(3/2)^{2m-1}K \rho^{2m-3}$.
The condition that $\dist(\bx,\balpha) \le 2\rho$ implies that 
$3\ge 1+\dist(\bx,\balpha)/\rho$, which handles the first case.

{\bf Case II:} We use the fact, from the definition of the surface spline (\ref{surface_spline}),
 that 
 $k(\bx,\bzeta) = (1 - \cos^2 \frac{\dist(\bzeta, \bx)}{2})^{s} $ 
 with 
 $2s = 2m-3$.
 Since $0\le\dist(\bzeta, \bx)\le \pi$, we can rewrite the kernel as
 $k(\bx,\bzeta)  =  \vartheta_s(t_{\bx}(\bzeta))$, with 
 $t_{\bx}(\bzeta):= \cos^2\bigl(\frac{\dist(\bzeta, \bx)}{2}\bigr) = \bigl(\tau_{\bx}{(\bzeta)}\bigr)^2$
 and $\vartheta_s := (1-\cdot)^s$.
 We note that 
 $\vartheta_s$  
 is smooth for $t$ less than $1$. 
Lemma \ref{kernel_cost} tells us that the error is controlled by $|\mathcal{I}_{\bx}|^{L+1}$. The length of $\mathcal{I}_{\bx}$ can be written as
\begin{eqnarray*}
 |t_{\bx}(\bxi) - t_{\bx}(\bzeta)| 
 &=& 
 \left|
   \cos^2\left(\frac{\dist(\bxi,\bx)}{2}\right) 
   - 
   \cos^2\left(\frac{\dist(\bzeta,\bx)}{2}\right)
\right|\\
&=&  
\left|
 \sin \left( \frac{\dist(\bxi, \bx)}{2}\right) 
 - 
 \sin\left(\frac{\dist(\bzeta, \bx)}{2}\right)
 \right| \times \\
&\mbox{}&
\left|\sin\left(\frac{\dist(\bxi, \bx)}{2}\right) +\sin\left(\frac{\dist(\bzeta, \bx)}{2}\right)\right|.
\end{eqnarray*}
Thus, for $\bx$ in a sufficiently wide annulus about $\balpha$, say 
$2^n \rho \le \dist(\bx,\balpha) \le 2^{n+1} \rho$, 
we have 
$2^{n-1} \rho \le \dist(\bx,\bxi) \le 2^{n+2} \rho$ 
for all $\bxi \in \supp{\a(\cdot, \balpha)}$, 
and we can estimate $|\mathcal{I}_{\bx}|$ by
$$
|\mathcal{I}_{\bx}| 
\le 
\frac{|\dist(\bzeta,\bxi)|}{2}\times \frac{|\dist(\bx,\bxi)+ \dist(\bx,\bzeta)|}{2}
\le \rho\times (2^{n+2}\rho).
$$
Applying Lemma \ref{kernel_cost} gives
$$
e_{k_m}(\bx,\balpha)
\le 
\frac{K}{(L+1)!} \|
 \vartheta_s^{(L+1)}  
\|_{L_\infty(\mathcal{I}_{\bx})} \rho^{2(L+1)} 2^{(n+2)(L+1)}. 
$$
Differentiating
gives
$
\|
\vartheta_s^{(L+1)}  
\|_{L_\infty(\mathcal{I}_x)}\le
\left|\prod_{j=0}^L (s-j)
\right|\max_{t\in \mathcal{I}_{\bx}} (1-t)^{s-L-1},
$
and we note that the maximum is attained at $t=t_{\bx}(\bzeta)= \cos^2\frac{\dist(\bzeta, \bx)}{2}$ 
for the rotation $\bzeta\in \{\balpha\}\cup\supp{\a(\cdot,\balpha)}$ nearest to $\bx$. 
Rewriting $1-t_{\bx}(\bzeta)$ as  $\sin^2\frac{\dist(\bzeta, \bx)}{2},$ 
and using the inequality 
$\sin\theta\ge\frac{2}{\pi}\theta$ 
gives
$1-t_{\bx}(\bzeta)\ge (\frac{2}{\pi}\frac{\dist(\bzeta, \bx)}{2})^2 \ge (\frac{2^{n-1}\rho}{\pi})^2$. 
Thus,
\begin{eqnarray*}
\|
 \vartheta_s^{(L+1)}  
\|_{L_\infty(\mathcal{I}_x)}
&\le& \left|s(s-1)\cdots(s-L)\right|\left(\frac{2^{2(n-1)} \rho^2}{\pi^2}\right)^{s-L-1}\\
&=& \left|s(s-1)\cdots(s-L)\right| \pi^{2(L+1-s)} \rho^{2(s-L-1)} 2^{2(n-1)(s-L-1)}.
\end{eqnarray*} 
Together with the estimate of the length of $|\mathcal{I}_{\bx}|$, we obtain 
\begin{eqnarray*}
e_{k_m}(\bx,\balpha)&\le&
K\frac{(8 \pi)^{2(L+1-s)}}{(L+1)!}|s(s-1)\cdots(s-L)|\rho^{2s}\,2 ^{(n+2)(2s-1-L)}.
\end{eqnarray*}
Because $2^{n+2}\ge1+ \frac{\dist(\zb x,\zb\alpha)}{\rho}$, the result follows.
\end{proof}

\section{ Approximation Scheme and Main Results}\label{s_main}
We are now ready to introduce the approximation scheme central to our error estimates. Given a coefficient kernel $\a$ satisfying the CKCs with precision $L\geq2m$, we define the operator $T_{\zb \Xi}$ on $W_p^{2m}\bigl(SO(3)\bigr)$ by
\begin{equation}\label{approximant}
T_{\zb \Xi} f(\bx) = 
\sum_{\bxi \in \zb \Xi} A_{\zb \xi} k_m(\bx,\bxi)
\end{equation}
with $A_{\zb \xi} :=   \int_{SO(3)} \L_m f(\balpha) \a(\bxi,\balpha)  \dif \mu(\balpha) $ (recalling that the operator $\L_m$ was defined in Lemma \ref{ss_is_Green}). 

The approximant 
$$T_{\zb \Xi} f := \int_{SO(3)} \sum \a(\bxi,\balpha) k_m(\cdot, \bxi) \L_{m}f(\balpha) \dif \mu(\balpha)$$ 
is the function obtained by replacing the kernel in the integral representation by linear combinations of its nearby shifts.

An apparent drawback of this approach is
the requirement that $f$ should be in the domain of $\L_{m}$: this is
a restrictive smoothness condition on the target function. To remedy this, we must modify the approach slightly, by splitting the (less smooth) target function into two parts:
a component of sufficient smoothness (in $L_p$)  and
a rough but negligible component.  The approximant to a general function will be obtained by applying the operator $T_{\zb\Xi}$  to its smooth part only. Precisely how this split is accomplished will be set aside for now, and revisited later in this section. 
\subsection{Functions of Full Smoothness}
We obtain error estimates by bounding the norm of the operator induced by the error kernel,
i.e.,  
$$E_{k_m}:L_p\bigl(SO(3)\bigr)\to L_p\bigl(SO(3)\bigr),\quad 
E_{k_m} g(\bx) = \int_{SO(3)} e_{k_m}(\bx,\balpha) g(\balpha)\dif\mu(\zb\alpha).$$ The bound on such an operator controls the error of the global approximant, as we will see in the following theorem.

\begin{theorem}\label{main:full}
Let $m\ge 2$, and let $\a:\zb\Xi\times SO(3) \to \reals$  be a coefficient kernel satisfying CKCs with radius $\rho$, precision $L\geq 2m$, and stability constant $K$.
For $f\in W_p^{2m}(SO(3))$, $1\le p<\infty$, the approximant converges, in $L_p$, to $f$ 
 with error bounded by
$$\|f - T_{\zb \Xi}f\|_p \le C \rho^{2m} \|f\|_{W_p^{2m}(SO(3))}.$$
Likewise, for $f\in C^{2m}(SO(3))$,  the approximant converges in $L_{\infty}$
with error
$$\|f - T_{\zb \Xi}f\|_{\infty} \le C \rho^{2m} \|f\|_{C^{2m}(SO(3))}.$$
\end{theorem}
\begin{proof}
For $1\le p\le \infty$ we can estimate the approximation error by
\begin{eqnarray*}
\|f - T_{\zb\Xi}f\|_p 
&= &
\left\| 
\int_{SO(3)} \L_m f(\balpha) \Big( k_m(\cdot,\balpha)-\; \sum_{\bxi \in \Xi} a(\bxi,\balpha) k_m(\cdot,\bxi)\Big)\dif \mu(\balpha) 
\right\|_p \\
&\le& 
\left\|\int_{SO(3)} |\L_m f(\balpha)| \; e_{k_m}(\cdot,\balpha)\, \dif \mu(\balpha)\right\|_p \le  \|E_{k_m}\|_{p\to p}\|\L_m f\|_p.
\end{eqnarray*}
We need to consider only two values of $p$, namely $p = 1$ and $\infty$, because the others follow by interpolation. 
The $L_1$ operator norm is controlled by 
$M_1:=\sup_{\balpha\in SO(3)}\int e_{k_m}(\bx,\balpha) \dif\mu(\bx)$.
Likewise, the $L_{\infty}$ norm is controlled by 
$M_{\infty} :=\sup_{\bx\in SO(3)}\int e_{k_m}(\bx,\balpha) \dif \mu(\balpha)$. 
Since the estimate from Lemma \ref{ss_loc} controls $e_{k,m}$ by a quantity  involving $\omega(\balpha^{-1}\bx)$ only, the error kernel is generated by a class function, and we can make a change of variable to see that the result is the same for $L_1$ and $L_{\infty}$ (and, hence, all other $1<p<\infty$ as well):
\begin{eqnarray*}
\max{(M_1,M_{\infty})}
&\le& 
\int_{SO(3)}\const \rho^{2m-3}\left(1+ \frac{\dist (\cdot,\balpha)}{\rho}\right)^{2m-4-L}\dif \mu(\balpha)\\
&\le&\const \rho^{2m-3}\int_0^{\pi} \left(1+ \frac{\omega}{\rho}\right)^{2m-4-L} \sin^2 \frac{\omega}{2}\, \dif \omega\\
&\le& \const \rho^{2m}\int_0^{\infty} \left(1+ \theta \right)^{2m-4-L}  \theta^2\, \dif \theta\leq C\, \rho^{2m}.
\end{eqnarray*}
 Interpolating over $1\le p\le\infty$ we obtain
$$\|E_{k_m}\|_{p\to p} \le C\, \rho^{2m},$$
and, hence, 
$\|f - T_{\zb \Xi}f\|_p\le   \|E_{k_m}\|_{p\to p} \|\L_m f\|_{p} 
\le 
C\, \rho^{2m}\|\L_m f\|_{p}.$
\end{proof}
\begin{note}\label{coeffs_note}
The size of the coefficients $A_{\zb \xi}$ can be estimated as follows. 
Given a coefficient kernel $\a$ satisfying the CKCs with  stability constant $K$, 
we have
$$\|A\|_{\ell_1(\zb \Xi)} \le K 
 \|\L_m f\|_1.$$
Indeed,
$$\sum_{\bxi\in \zb \Xi}|A_{\zb \xi}| \le \int_{SO(3)} \sum_{\bxi\in\zb  \Xi} |\a(\balpha,\bxi)| |\L_m f (\balpha) | \dif \mu(\balpha) \le K \int_{SO(3)}  |\L_m f (\balpha) | \dif \mu(\balpha).$$
\end{note}
Applying Lemma \ref{M_Z} to the previous theorem gives the following corollary.
%
%
\begin{corollary}
For  centers $\zb \Xi$ having density $h<h_0$ (with $h_0$ determined by
$L$  as in Lemma \ref{M_Z}), there exists a coefficient kernel such that the approximant defined in \eqref{approximant} converges like
$$\|f - T_{\zb \Xi}f\|_p \le C h^{2m} \|f\|_{W_p^{2m}(SO(3))}$$
with coefficients satisfying the estimate
$$\|A\|_{\ell_1(\zb \Xi)} \le 2  \|\L_m f\|_1.$$
\end{corollary}

\subsection{Functions of Lower Smoothness}
To treat more general functions, we make use of a common $K$-functional argument, one practically ubiquitous in approximation theory. This will permit us to define approximants for functions in Besov spaces of lower order (and ultimately get good error bounds for such functions as well).

\begin{definition}For $0<s\le 2m$ and $1\le p< \infty$, we define the Besov space $B_{p,\infty}^s\bigl(SO(3)\bigr)$ as the collection of 
functions in $L_p\bigl(SO(3)\bigr)$  for which the following expression 
$$\|f\|_{B_{p,\infty}^s}:=\sup_{t>0} t^{-s} K(f,t) $$
is finite, where the $K$-functional $K(f,\cdot):(0,\infty)\to (0,\infty)$ is defined as
$$K(f,t) := \inf\left\{\|f - g\|_{L_p} + t^{2m}\|g\|_{W_p^{2m}}: g\in W_p^{2m} \right\}.$$
For $p=\infty$, the definition is the same after substituting
 $L_{p}\bigl(SO(3)\bigr)$ by $C\bigl(SO(3)\bigr)$ and $W_p^{2m}$ by $C^{2m}$.
\end{definition}
There are abundant alternative characterizations of Besov spaces. 
We have chosen the preceding because it  captures precisely the properties
we need from a function of lower smoothness. 
It is possible to follow the example of Definition \ref{Sobolev}, by transporting the 
Besov norm from $\reals^3$. 
For a general manifold this would not work, it is possible only because $SO(3)$ is compact 
(specifically, it works because the diffeomorphisms and partitions of unity  employed 
have finite cardinality-- see \cite[1.11.3, Remark 2 and 2.4.7]{Trieb} for 
a discussion of this). 
We recommend \cite[1.11, and Chapter 7]{Trieb} and the references therein for background
on Besov spaces on Riemannian manifolds.

An immediate consequence of the definition is that, for an arbitrary function in $B_{p,\infty}^s$, and ${\rho}>0$, there is $g_{\rho}\in W_p^{2m}$ simultaneously satisfying
\begin{eqnarray}
&\mbox{}& \|g_{\rho}\|_{W_p^{2m}} \le 2 {\rho}^{s-2m}\|f\|_{B_{p,\infty}^s},\label{Bern}\\
&\mbox{}& \|f-g_{\rho}\|_{L_p} \le 2 {\rho}^s \|f\|_{B_{p,\infty}^s}.\label{Jack}
\end{eqnarray}
In the setting of Theorem \ref{main:full} (namely, with radius $\rho$), the approximant to $f$ is then $s_{f,\zb \Xi}:=T_{\zb \Xi}g_{\rho}$.

\begin{note} It is important to point out that the function $g_{\rho}$ can be constructed
by the same method employed in Section \ref{ss_smoothness_spaces} to create the $C^{\infty}$ approximant. Namely,
one may employ a $C^{\infty}$ partition of unity $(\psi_j)_{j=1}^N$, a corresponding sequence of smooth diffeomorphisms, and a sequence of cut-offs $(\tilde{\psi}_j)_{j=1}^N$ to generate an approximant of the form
$$g_{\rho} = \sum_{j = 1}^N \tilde{\psi}_j \times (g_{j,\rho} \circ h_j)$$
where $g_{j,\rho}:\reals^3\to \reals$ approximates the function $f_j = (f\times\psi_j)\circ h_j^{-1}\in B_{p,\infty}^s(\reals^3)$ in the sense of (\ref{Bern}) and (\ref{Jack}). 
The Euclidean approximants $g_{j,\rho}$ can be constructed,
for example, by the method \cite[Theorem 5.33]{Adams2} or  \cite[Lemma 5.1]{DeRo}.
\end{note}

\begin{note}
 Again, we can estimate the size of the coefficients. For $f\in B_{p,\infty}^s$ we have
$$\|A\|_{\ell_1(\zb \Xi)}\le 2 K  {\rho}^{s-2m} \| f\|_{B_{p,\infty}^s}.$$
This follows directly from Note \ref{coeffs_note}.
\end{note}
\begin{theorem}\label{main:lower}
Let $s<2m$ and $1\le p\le \infty$. Suppose the coefficient kernel $\a:\zb\Xi\times SO(3) \to \reals$ satisfies CKCs with radius $\rho$,  precision $L\ge 2m$, and stability constant $K$. Then for $f\in B_{p,\infty}^s$ there is $s_{f,\zb \Xi}$ satisfying
$$\|f - s_{f,\zb\Xi}\|_p \le C \rho^{s} \|f\|_{B_{p,\infty}^{s}}.$$
\end{theorem}
\begin{proof}
This follows by applying Theorem \ref{main:full} to $g_{\rho}$ and observing that (\ref{Bern}) ensures
$\|\L_m g_{\rho}\|_p\lesssim \rho^{s-2m} \|f\|_{B_{p,\infty}^{s}}$.
\end{proof}
Applying Lemma \ref{M_Z} again, we obtain:
\begin{corollary}
For  centers $\zb\Xi$ having density $h<h_0$ (with $h_0$ determined by
$L$  as in Lemma \ref{M_Z}), and for $f\in B_{p,\infty}^s$, there is an
approximant $s_{f,\zb\Xi}$ satisfying
$$\|f - s_{f,\zb\Xi}\|_p \le C h^{s} \|f\|_{B_{p,\infty}^{s}}$$
with coefficients satisfying the estimate
$$\|A\|_{\ell_1(\zb \Xi)} \le 2 h^{s-2m} \|f\|_{B_{p,\infty}^{s}}.$$
\end{corollary}

\section{Concluding Remarks}
The  results of the last section have value beyond solving a challenging (if theoretical) problem,
and even beyond
providing benchmarks for other approximation schemes.  
When approximation schemes based on projection 
(e.g., interpolation and least-squares projection, discussed in \ref{ss_practical}) are {\em stable}, 
application of a {\em Lebesgue lemma} implies that such schemes are {\em near best}. 
Very recently,  in \cite{HNW},  it has been demonstrated that kernel interpolation on smooth, compact Riemannian manifolds is stable for a certain family of positive definite kernels. 
Under similar circumstances in \cite{HNSW}, least-squares projection has been shown to be 
stable in $L_p$ for each $1\le p\le \infty$.
In the setting of the 2-sphere, $\mathbb{S}^2,$ it is known that these kernels provide precise (if theoretical) $L_{p}$ rates of approximation. 
Hence, interpolation on $\mathbb{S}^2$ (an easily implemented approximation scheme) 
and least-squares approximation
inherit these precise rates for all levels of smoothness up to a saturation order of the kernel.

The kernels employed in \cite{HNW,HNSW} have variants on $SO(3)$ as well, 
and they are very similar to the kernels considered here. Both families are composed of fundamental solutions of elliptic differential operators which serve as reproducing kernels for the Sobolev spaces $W_2^m\bigl(SO(3)\bigr)$ (albeit for different, but equivalent inner products). 
This raises the strong likelihood that similar stability results hold for the kernels considered here, or that, by modifying the techniques of this article,
precise error estimates can be shown for the kinds of kernels considered in \cite{HNW, HNSW}.
 The upshot is that precise error estimates for `best' approximation may lead directly to  precise error estimates for any stable linear approximation scheme obtained by projection.

In addition to this, the approach we take  has the potential to work in far greater generality: specifically on other structures like Lie groups and homogeneous spaces. It relies on two basic properties that, at first, seem to have little to do with groups, but rather with (compact) manifolds. These are: 
\begin{itemize}
\item the representation (\ref{rep}) involving the fundamental solution for a differential operator: iterated, perturbed Laplace \!--\! Beltrami operators; 
\item the ability to efficiently replace this fundamental solution by a linear combination of nearby translates incurring little error (Lemma \ref{ss_loc}).
\end{itemize}
Finding kernels that satisfy both of these simultaneously is a challenge; the second property
seems especially difficult without some extra structure. However, when the manifold is a group or a homogeneous space, it is often possible to exploit some underlying symmetry. In this article we considered kernels, 
$\kappa(\bx,\balpha)$, derived from `class functions', i.e. functions depending on 
$\balpha^{-1}\bx$ -- meaning $\kappa(\bx,\balpha) = f(\balpha^{-1}\bx)$ -- but possessing an added symmetry;
the function $f$ depends only on the rotational angle of $\balpha^{-1}\bx$ (which on $SO(3)$ is the same as the Riemannian distance). It is precisely this symmetry --
ultimately embodied in the addition formula which relates the basic elements of harmonic analysis on the space (the eigenfunctions of the Laplace \!--\! Beltrami operator) to an orthonormal basis for the space of class functions, and which, in turn, can be related to certain orthogonal polynomials --
that ultimately allowed us to effectively replace the kernel.

\bibliographystyle{siam}
\bibliography{SO3_acha_revision}
%
%
%

%
%
\end{document}